\documentclass{article}
\usepackage[utf8]{inputenc}

\usepackage{mystyle}

\title{Reducing the domination number of $P_3+kP_2$-free graphs via one edge contraction}
\author[1]{E. Galby}
\author[2]{F. Mann}
\author[2]{B. Ries}
\affil[1]{CISPA Helmholtz Center for Information Security, Saarbr\"ucken, Germany}
\affil[2]{University of Fribourg, Department of Informatics, Fribourg, Switzerland}
\date{}

\begin{document}

\maketitle

\begin{abstract}
    In this note, we consider the following problem: given a connected graph $G$, can we reduce the domination number of $G$ by using only one edge contraction? We show that the problem is polynomial-time solvable on $P_3+kP_2$-free graphs for any $k \geq 0$ which combined with results of \cite{beith,dagstuhl} leads to a complexity dichotomy of the problem on $H$-free graphs.
\end{abstract}

\section{Introduction}

Given a graph $G=(V(G),E(G))$, a \emph{dominating set} of $G$ is a subset $D \subseteq V(G)$ of vertices such that every vertex in $V(G) \setminus D$ has a neighbor in $D$, and $D$ is minimum if it has minimum cardinality amongst all dominating sets of $G$. The \emph{domination number} of $G$ is the cardinality of a minimum dominating set of $G$. The \emph{contraction} of an edge $uv \in E(G)$ removes vertices $u$ and $v$ from $G$ and replaces them with a new vertex which is made adjacent to precisely those vertices that were adjacent to $u$ or $v$ in $G$ (without introducing self-loops nor multiple edges). In this note, we consider the following problem.

\begin{center}
\begin{boxedminipage}{.99\textwidth}
\textsc{1-Edge Contraction($\gamma$)}\\[2pt]
\begin{tabular}{ r p{0.8\textwidth}}
\textit{~~~~Instance:} &A connected graph $G$.\\
\textit{Question:} &Does there exist an edge $e\in E(G)$ such that contracting $e$ reduces the domination number by at least one?
\end{tabular}
\end{boxedminipage}
\end{center}

It was shown in \cite{beith} that the problem is $\mathsf{NP}$-hard in general graphs. As a  consequence, the authors considered restrictions of the input to special graph classes and proved in particular the following.

\begin{theorem}[\cite{beith}, Corollary 1.4]
\contracd is $\mathsf{NP}$- or $\mathsf{coNP}$-hard on $H$-free graphs if $H$ is not a linear forest\footnote{A linear forest is a forest of maximum degree 2, or equivalently, a disjoint union of paths.} or contains an induced $P_6$, $P_4+P_2$ or $2P_3$, and solvable in polynomial time if $H$ is an induced subgraph of $P_5+qP_1$ for some $q\geq 0$.
\end{theorem}

To obtain a complexity dichotomy of \contracd on $H$-free graphs, there only remains to settle the complexity of the problem when $H$ is an induced subgraph of $P_3+kP_2+qP_1$ for some $k \geq 1$ and $q\geq 0$. In this paper we solve this remaining case and prove the following. 

\begin{theorem}\label{thm:P3kP2}
\contracd is polynomial-time solvable on $P_3+kP_2$-free graphs for any $k \geq 0$.
\end{theorem}

Combined with results of \cite{beith,dagstuhl}, we obtain the following dichotomy.

\begin{theorem}\label{thm:dichotomy}
\contracd is polynomial-time solvable on $H$-free graphs if and only if $H$ is an induced subgraph of $P_5+qP_1$ for some $q\geq 0$ or an induced subgraph of $P_3+kP_2+qP_1$ for some $k,q\geq 0$, unless $\mathsf{P} = \mathsf{NP}$.
\end{theorem}

\section{Preliminaries}

Throughout this paper the considered graphs are finite, simple and connected, unless stated otherwise. 

For $n\geq 1$, the path on $n$ vertices is denoted by $P_n$. If $G$ is a graph and $k\in\mathbb{N}$, then we denote by $kG$ the graph consisting of $k$ disjoint copies of $G$. 

Given a graph $G$, we denote by $V(G)$ its vertex set and by $E(G)$ its edge set. The \textit{(open) neighbourhood} $N(v)$ of a vertex $v\in V(G)$ is the set $\set{w\in V(G)\colon\, vw\in E(G)}$. The \textit{closed neighbourhood} $N[v]$ of a vertex $v\in V(G)$ is the set $N(v)\cup\set{v}$. If $S\subseteq V(G)$ and $v\in V(G)$ then we say that $v$ is \textit{complete} to $S$ if $v$ is adjacent to every vertex in $S$. For $S\subseteq V(G)$ we write $G[S]$ for the graph \textit{induced by $S$}, that is, the graph with vertex set $V(G[S])=S$ and edge set $E(G[S])=\set{xy\in E(G)\colon\, x,y\in S}$. A set $S\subseteq V(G)$ is called a \textit{clique} (respectively a \textit{stable set}) if every two vertices in $S$ are adjacent (respectively non-adjacent). For two vertices $x,y\in V(G)$ the \textit{distance $d(x,y)$ from $x$ to $y$} is the number of edges in any shortest path between $x$ and $y$. Given a dominating set $D$ of $G$ and a vertex $u\in D$, every vertex $v\in V(G)$ such that $N[v]\cap D=\set{u}$ is called a \textit{private neighbour of $u$}.

The following theorem characterises the \yes-instances for \contracd.
\begin{theorem}[\cite{HuangXu}, Lemma 3.5]\label{NOInstStable}
A graph $G$ is a \yes-instance for \contracd if and only if there exists a minimum dominating set of $G$ which is not a stable set.
\end{theorem}

We will also use the following theorem which presents some cases where \contracd is polynomial-time solvable.

\begin{theorem}[\cite{dagstuhl}, Proposition 12]\label{boundedgamma}
\contracd can be solved in polynomial time for a graph class $\mathcal{C}$ if for every graph $G\in\mathcal{C}$ we have $\gamma(G)\leq q$, where $q$ is some fixed constant. If $H$ is a graph and \contracd can be solved in polynomial time for $H$-free graphs, then \contracd\ can be solved in polynomial time for $(H+K_1)$-free graphs.

\end{theorem}

\section{Proof of \Cref{thm:P3kP2}}

First observe that if a graph $G$ does not contain an induced $P_3$ then $G$ is a clique and thus a \no-instance for \contracd. Assume henceforth that $k \geq 1$ and let $G$ be a $P_3+kP_2$-free containing an induced $P_3+(k-1)P_2$. Let $A\subseteq V(G)$ be such that $G[A]$ is isomorphic to $P_3+(k-1)P_2$, let $B\subset V(G)$ be the set of vertices at distance one from $A$ and let $C\subset V(G)$ the set of vertices at distance two from $A$. Note that since $G$ is $P_3+kP_2$-free, the sets $A,B$ and $C$ partition $V(G)$ and $C$ is a stable set. Denote by $\mathcal{C}\subseteq C$ the set of vertices whose neighbourhoods are cliques. We call a vertex $v_1\in\mathcal{C}$ a \emph{regular vertex} if there exist $k$ vertices $v_2,\ldots,v_{k+1}\in\mathcal{C}$ such that $v_1,\ldots,v_{k+1}$ are pairwise at distance at least four from one another. We denote by $\mathcal{R}$ the set of regular vertices.

\begin{claim}\label{VertexCompleteness}
Let $c_1\in\mathcal{R}$ be a regular vertex. If a vertex $v\not\in N[c_1]$ is adjacent to a vertex in $N(c_1)$ then there exists a regular vertex $c\in\mathcal{R}$ such that $v$ is complete to $N(c)$.
\end{claim}
\begin{proof}
Let $c_1,\ldots,c_{k+1}\in\mathcal{R}$ be $k+1$ regular vertices which are pairwise at distance at least four from one another. Suppose for a contradiction that $v$ is adjacent to $b_1\in N(c_1)$ and for every $j\in [k+1] \setminus \{1\}$, there exists a vertex $b_j\in N(c_j)$ such that $v$ is not adjacent to $b_j$. Then $\bigcup_{i=1}^{k+1}\set{b_i,c_i}\cup\set{v}$ induces a $P_3+kP_2$, a contradiction.
\end{proof}

\begin{claim}\label{OneVertexInDPerRegularVertex}
Let $D$ be a minimum dominating set of $G$ and let $c_1\in\mathcal{R}$ be a regular vertex. Then $|D\cap N[c_1]|=1$.
\end{claim}
\begin{proof}
Let $c_1,\ldots,c_{k+1}\in\mathcal{R}$ be $k+1$ regular vertices which are pairwise at distance at least four from one another. As for any $i\in[k+1]$ $c_i$ must be dominated, \[\left\vert D\cap \bigcup_{i=1}^{k+1}N[c_i]\right\vert\geq k+1.\] 
Suppose for a contradiction that $\vert D\cap N[c_1]\vert\geq 2$. For every $i\in [k+1]$, let $b_i\in B$ be a vertex adjacent to $c_i$. If a vertex $v\in V(G)$ is adjacent to $D\cap N[c_i]$ for some $i\in[k+1]$ then either $v\in N[c_i]$ in which case $v$ is adjacent or identical to $b_i$, or $v\in V(G)\setminus N[c_i]$ and by \Cref{VertexCompleteness} there exists $j\in[k+1]$ such that $v$ is complete to $N(c_j)$; in particular, $v$ is then adjacent to $b_j$. It now follows that $\left(D\setminus\bigcup_{i\in[k+1]}N[c_i]\right )\cup \set{b_1,\ldots,b_{k+1}}$ is a dominating set of $G$ of cardinality less than $\vert D\vert$, a contradiction to the minimality of $D$.
\end{proof}

By using similar arguments as in the proof of \Cref{OneVertexInDPerRegularVertex}, we can prove the following.
\begin{corollary}\label{RegularNeighbourhoodShift}
Let $c_1,\ldots,c_{k+1} \in \mathcal{R}$ be $k+1$ regular vertices which are pairwise at distance at least four from one another. For every $i\in[k+1]$, let $b_i$ be a neighbour of $c_i$. If $D$ is a minimum dominating set of $G$ then $\left(D\setminus\bigcup_{i\in[k+1]}N[c_i]\right)\cup\set{b_1,\ldots,b_{k+1}}$ is a minimum dominating set of $G$ as well.
\end{corollary}

\begin{claim}\label{AllRegularVerticesAtDistanceFour}
No two regular vertices have a common neighbour. Furthermore, if there are two regular vertices $c_1,c'_1$ at distance three from each other then $G$ is a \yes-instance for \contracd.
\end{claim}
\begin{proof}
Suppose that there exist two regular vertices $c_1,c'_1\in\mathcal{R}$ at distance at most three from one another. Let $c_2,\ldots,c_{k+1}\in\mathcal{R}$ (respectively $c'_2,\ldots,c'_{k+1}$) be $k$ regular vertices such that $c_1,\ldots,c_{k+1}$ (respectively $c'_1,\ldots,c'_{k+1}$) are pairwise at distance at least four from one another. For every $i\in[k+1]\setminus\set{1}$, let $b_i$ (respectively $b'_i$) be a neighbour of $c_i$ (respectively $c'_i$) such that $b_i=b'_j$ whenever $c_i=c'_j$. Suppose for a contradiction that $d(c_1,c'_1)= 2$ and let $b_1\in B$ be a common neighbour of $c_1$ and $c'_1$. Observe that $c'_1$ cannot be adjacent to $b_i$ for any $i\geq 2$, since the neighbourhood of $c'_1$ is a clique and $b_i$ is not adjacent to $b_1$ (recall that $c_1,\ldots,c_{k+1}$ are pairwise at distance at least four from one another). It follows that the vertices $c'_1,c_1,\ldots,c_{k+1},b_1,\ldots,b_{k+1}$ induce a $P_3+kP_2$, a contradiction. Suppose now that $d(c_1,c'_1)=3$ and let $b_1\in N(c_1)$ and $b'_1\in N(c'_1)$ be two adjacent vertices. Consider a minimum dominating set $D$ of $G$. Then by \Cref{RegularNeighbourhoodShift} $\left(D\setminus\bigcup_{i\in[k+1]}N[c_i]\cup N[c'_i]\right)\cup\set{b_1,b'_1,\ldots,b_{k+1},b'_{k+1}}$ is a minimum dominating set of $G$ containing an edge; \Cref{NOInstStable} then implies the claim. 
\end{proof}

\begin{claim}\label{OnlyOnePNInC}
Assume that $G$ is a \no-instance for \contracd and let $D$ be a minimum dominating set of $G$. If there exists a vertex $b_0\in B\cap D$ which has more than one private neighbour in $C$ then $|B\cap D|\leq k|A|$.
\end{claim}
\begin{proof}
Observe first that $D$ is a stable set by \Cref{NOInstStable}. Assume that there exists a vertex $b_0\in B\cap D$ which has at least two private neighbours in $C$, say $x$ and $y$. Suppose for a contradiction that there are at least $k|A|$ further vertices in $B\cap D$ besides $b_0$, say $b_1,\ldots,b_{k|A|}$. We claim that for every $i\in[k]$ there exists a vertex $c_i\in C$ such that $N[c_i]\cap D\subseteq\set{b_{(i-1)|A|+1},\ldots,b_{i|A|}}$. Indeed, if for some $i \in [k+1]$ there is no such vertex in $C$ then $\left(D\setminus\set{b_{(i-1)|A|+1},\ldots,b_{i|A|}}\right)\cup A$ is a minimum dominating set of $G$ containing an edge, a contradiction to \Cref{NOInstStable}. Now assume, without loss of generality, that $b_{i|A|}$ is adjacent to $c_i$ for every $i\in[k]$. Then the vertices $b_0,x,y,c_1,\ldots,c_k,b_{|A|},b_{2|A|},\ldots,b_{k|A|}$ induce a $P_3+kP_2$, a contradiction.
\end{proof}

\begin{claim}\label{SharedResponsibilitiesAreDistributed}
Assume that $G$ is a \no-instance for \contracd and let $D$ be a minimum dominating set. If there exists a vertex $c\in C$ such that $\vert N(c)\cap D\vert\geq 2$ then $c$ is adjacent to all the vertices in $B\cap D$ except for at most $k|A|-1$.
\end{claim}
\begin{proof}
Observe first that $D$ is a stable set by \Cref{NOInstStable}. Assume that $c\in C$ has at least two neighbours in $B\cap D$, say $x$ and $y$. Suppose for a contradiction that there are at least $k\vert A\vert$ vertices in $B\cap D$ which are not adjacent to $c$, say $b_1,\ldots,b_{k|A|}$. As shown in the proof of \Cref{OnlyOnePNInC}, there has to be for every $i\in[k]$ a vertex $c_i$ such that $N[c_i]\cap D\subseteq \set{b_{(i-1)|A|+1},\ldots,b_{i|A|}}$. Assume, without loss of generality, that $b_{i|A|}$ is adjacent to $c_i$ for every $i\in[k]$. Then the vertices $c,x,y,c_1,\ldots,c_k,b_{|A|},b_{2|A|},\ldots,b_{k|A|}$ induce a $P_3+kP_2$, a contradiction.
\end{proof}

\begin{corollary}\label{FewVerticesInCWithoutPN}
Assume that $G$ is a \no-instance for \contracd and let $D$ be a minimum dominating set of $G$. If there are at least $|A|$ vertices in $B\cap D$ which do not have a private neighbour in $C$ then $|B\cap D|\leq (k+1)|A|-1$.
\end{corollary}
\begin{proof}
Assume that $\vert B\cap D\vert\geq (k+1)\vert A\vert$. Suppose for a contradiction that there are at least $|A|$ vertices in $B \cap D$, say $b_1,\ldots,b_{|A|}$, which have no private neighbours in $C$. Then for every $i\in[\vert A\vert]$, any vertex $c\in N(b_i)\cap C$ has to be adjacent to at least two vertices in $B\cap D$ (note indeed that by \Cref{NOInstStable} $c$ does not belong to $D$) and thus by \Cref{SharedResponsibilitiesAreDistributed}, $c$ has to be adjacent to at least $\vert A\vert+1$ vertices in $B\cap D$. But then $(D\setminus\set{b_1,\ldots,b_{|A|}})\cup A$ is a minimum dominating set of $G$ containing an edge, a contradiction to \Cref{NOInstStable}.
\end{proof}

\begin{claim}\label{BCapDIsSmall}
Assume that $G$ is a \no-instance for \contracd and let $D$ be a minimum dominating set of $G$. If there exists a vertex $v\in B\cap D$ which has a private neighbour $c\in N(v)\cap C$ and a private neighbour $b\in N(v)$ such that $c$ is not adjacent to $b$ then $|B\cap D|\leq (k+1)|A|$.
\end{claim}
\begin{proof}
If a vertex in $B\cap D$ has two private neighbours in $C$ then we conclude by \Cref{OnlyOnePNInC}. Thus, we can assume that no vertex in $B\cap D$ has more than one private neighbour in $C$. Assume that $v\in B\cap D$ has exactly one private neighbour $c\in C$ and assume further that $v$ has a private neighbour $b\in B$ such that $b$ and $c$ are not adjacent. Suppose for a contradiction that $|B\cap D|\geq (k+1)|A|+1$. Then by \Cref{FewVerticesInCWithoutPN} there are at most $|A|-1$ vertices in $B\cap D$ which do not have a private neighbour in $C$. Hence, besides $v$, there are at least $k\vert A\vert+1$ further vertices in $B\cap D$ which do have private neighbours in $C$. Let $b_1,\ldots,b_{|A|+k}\in B\cap D$ be $\vert A\vert+k$ such vertices with private neighbours $c_1,\ldots,c_{|A|+k}\in C$, respectively. By the pigeonhole principle, there are either $k$ indices $i\in[\vert A\vert+k]$ such that $c_i$ is non-adjacent to $b$ or $|A|+1$ indices $i\in[\vert A\vert+k]$ such that $c_i$ is adjacent to $b$. In the first case, assume, without loss of generality, that $c_1,\ldots,c_k$ are non-adjacent to $b$. Then the vertices $c,v,b,b_1,\ldots,b_k,c_1,\ldots,c_k $ induce a $P_3+kP_2$ (recall that by \Cref{NOInstStable} $D$ is a stable set), a contradiction. In the second case, assume, without loss of generality, that $b$ is complete to $\set{c_1,\ldots,c_{\vert A\vert+1}}$. Then by \Cref{SharedResponsibilitiesAreDistributed}, every vertex in $C$ which is adjacent to a vertex in $\set{b_1,\ldots,b_{\vert A\vert+1}}$ is adjacent to a vertex in $\left(\left(B\cap D\right)\setminus\set{b_1,\ldots,b_{\vert A\vert+1}}\right)\cup\set{b}$ as well. Thus, $\left(D\setminus\set{b_1,\ldots,b_{|A|+1}}\right)\cup\set{b}\cup A$ is a minimum dominating set of $G$ containing an edge, a contradiction to \Cref{NOInstStable}.
\end{proof}

\begin{claim}\label{MostVerticesInC}
If $G$ is a \no-instance for \contracd then there exists a minimum dominating set $D$ of $G$ such that $|D\setminus C|\leq (k+2)|A|$.
\end{claim}
\begin{proof}
Let $D$ be a minimum dominating set of $G$ such that $\vert B\cap D\vert$ is minimal amongst all minimum dominating sets of $G$. If $\vert B\cap D\vert\geq (k+1)\vert A\vert +1$, then by \Cref{OnlyOnePNInC} every vertex in $B\cap D$ has at most one private neighbour in $C$ and \Cref{FewVerticesInCWithoutPN} ensures that there is at least one vertex $b\in B\cap D$ which does have a private neighbour $c\in C$. But now either $(D\setminus\set{b})\cup\set{c}$ is a minimum dominating set of $G$, contradicting the fact that $\vert B\cap D\vert$ is minimal amongst all minimum dominating sets of $G$, or $b$ has a private neighbour $p\in N(b)$ which is not adjacent to $c$, a contradiction to \Cref{BCapDIsSmall}.
Hence $\vert B\cap D\vert\leq (k+1)\vert A\vert$ and since $\vert D\cap A\vert\leq\vert A\vert$ the claim follows.
\end{proof}

\begin{claim}\label{MostNeighbourhoodsAreCliques}
Assume that $G$ is a \no-instance for \contracd and let $D$ be a minimum dominating set of $G$. If $S\subseteq C\cap D$ is a subset of vertices which are pairwise at distance at least three from one another and every vertex in $S$ has two non-adjacent neighbours then $\vert S\vert\leq (k+1)^2-1$.
\end{claim}
\begin{proof}
Assume first that $S'=\set{c_1,\ldots,c_{k+1}}\subseteq C\cap D$ is a set of $k+1$ vertices which are pairwise at distance at least three from one another and for every $i\in[k+1]$ there are two non-adjacent vertices $b_i, b'_i\in N(c_i)$. If for every $i,j\in[k+1]$ the vertices $c_i$ and $c_j$ were at distance at least four then the vertices $b_1,b'_1,c_1,\ldots,b_{k+1},b'_{k+1},c_{k+1}$ would induce a $(k+1)P_3$, a contradiction. Hence there are two indices $i,j\in[k+1]$ such that $c_i$ and $c_j$ are at distance exactly three from one another. 

Now suppose for a contradiction that there is a set $S\subseteq C\cap D$ of at least $(k+1)^2$ vertices which are pairwise at distance at least three from one another and such that for every vertex $v\in S$ there are two vertices in $N(v)$ which are not adjacent. By the above, there must exist two vertices in $S$ at distance exactly three. Let $S_1\subseteq N(S)$ be a maximum subset of $N(S)$ such that $G[S_1]$ contains exactly one edge and no two vertices in $S_1$ share a common neighbour in $S$. Observe that $\vert N(S_1)\cap S\vert=\vert S_1\vert$ and that $S_1\cup (N(S_1)\cap S)$ induces a $P_4+(\vert S_1\vert-2)P_2$. This implies in particular that $\vert S_1\vert\leq k+1$. We construct a sequence of sets of vertices according to the following procedure.

\begin{itemize}
    \item[1.] Initialize $i=1$. Set $C_1=N(S_1)\cap S$ and $B_1=N(C_1)$.
    \item[2.] Increase $i$ by one.
    \item[3.] Let $S_i\subset N(S)\setminus B_{i-1}$ be a maximum set of vertices such that $G[S_i]$ contains exactly one edge and no two vertices in $S_i$ share a common neighbour in $S$. Set $C_i=C_{i-1}\cup (N(S_i)\cap S)$ and $B_i=B_{i-1}\cup N(C_i)$.
    \item[4.] If $\vert S_i\vert=\vert S_{i-1}\vert$, stop the procedure. Otherwise, return to step 2.
\end{itemize}

Consider the value of $i$ at the end of the procedure (note that $i\geq 2$). Observe that since for any $j \in [i-1] \setminus \{1\}$, $|S_j| < |S_{j-1}|$ and $|S_1| \leq k + 1$, it follows that for any $j \in [i-1]$, $|S_j| \leq k + 2 - j$. Let us show that $|S_i| \geq 2$. Since for any $j \in [i-1]$, $|S_j| \leq k + 1$, we have that $|S \setminus C_j| = |S| - \sum_{p = 1}^j |S_p| \geq (k+1)^2 - j(k+1)$. Thus if $i \leq k+1$ then for any $j \in [i-1]$, $|S \setminus C_j| \geq k+1$ which implies by the above that $|S_j| \geq 2$ for any $j \in [i]$. We now claim that $i$ cannot be larger than $k+1$. Indeed, if $i > k+1$ then for any $j \in [k+1] \setminus {1}$, $|S_j| < |S_{j-1}|$ with $|S_{k+1}| \geq 2$ as shown previously; but $|S_j| \leq k + 2 - j$ for any $j \in [i-1]$ which implies that $|S_{k+1}| \leq 1$, a contradiction. Thus $i\leq k+1$ and so $|S_i| \geq 2$.

Now observe that for any vertex $c\in N(S_i)\cap S$, every neighbour $v\in N(c)$ has to be adjacent to $S_{i-1}$ as otherwise the procedure would have output $S_{i-1}\cup\set{v}$ instead of $S_{i-1}$. Furthermore, for any vertex $c\in N(S_{i-1})\cap S$ every neighbour $v\in N(c)$ has to be adjacent to a vertex in $S_{i}$ as otherwise the procedure would have output $S_i\cup\set{v}$ instead of $S_{i-1}$ (recall that $\vert S_i\vert=\vert S_{i-1}\vert$). It follows that $\left(D\setminus \left(N\left(S_i\cup S_{i-1}\right)\cap S\right)\right)\cup\left(S_i\cup S_{i-1}\right)$ is a minimum dominating set of $G$ containing an edge, a contradiction to \Cref{NOInstStable}.
\end{proof}

\begin{claim}\label{AlmostAllAreDistanceThree}
Assume that $G$ is a \no-instance for \contracd and let $D$ be a minimum dominating set of $G$. Then the number of vertices in $C\cap D$ which are at distance two from another vertex in $C\cap D$ is at most $2|A|+(k+1)^2-3$.
\end{claim}
\begin{proof}
If every two vertices in $C\cap D$ are at distance at least three from one another then we are done. Thus assume that there are two vertices in $C\cap D$ which are at distance two from one another. Let $\mathcal{S}=\arg\max_{S\subseteq B}\vert N(S)\cap C\cap D\vert-\vert S\vert$ and let $S\in\mathcal{S}$ be a set of minimum size in $\mathcal{S}$. Note that since there are two vertices in $C\cap D$ which have a common neighbour, $S$ is non-empty. If $|N(S)\cap C\cap D|\geq |A|+|S|$ then $(D\setminus (N(S)\cap C\cap D)) \cup S\cup A$ is a dominating set of $G$ of cardinality at most $\vert D\vert$ which contains an edge, a contradiction to \Cref{NOInstStable}. Hence $\vert N(S)\cap C\cap D\vert<\vert S\vert+ \vert A\vert$. We now claim that every vertex in $S$ is adjacent to two vertices in $C\cap D$ which are not adjacent to any other vertex in $S$. Indeed, if a vertex $s\in S$ has no neighbour in $C\cap D$ which is not adjacent to any other vertex in $S$ then we could remove $s$ from $S$ without changing the cardinality of $N(S)\cap C\cap D$, thereby contradicting the fact that $S \in \mathcal{S}$. If a vertex $s\in S$ has only one neighbour $c$ in $C\cap D$ which is not adjacent to any other vertex in $S$ then removing $s$ from $S$ would only remove $c$ from $N(S)\cap C\cap D$, thus leaving the value of $|N(S)\cap C\cap D|-|S|$ unchanged while decreasing the cardinality of $S$, a contradiction to minimality of $\vert S\vert$. This implies in particular that $\vert N(S)\cap C\cap D\vert\geq 2\vert S\vert$ which combined with the inequality above leads to $|S|<|A|$ and $|N(S)\cap C\cap D|\leq 2|A|-2$. Now denote by $C'=(C\cap D)\setminus N(S)$ the set of all vertices in $C\cap D$ which are not adjacent to a vertex in $S$. Observe that no wo vertices $c,c'\in C'$ can have a common neighbour $b$, as otherwise $S'=S\cup\set{b}$ would be such that $\vert S'\vert=\vert S\vert+1$ and $\vert N(S')\cap C\cap D\vert\geq \vert N(S)\cap C\cap D\vert +2$ and thus $\vert N(S')\cap C\cap D\vert-\vert S'\vert>\vert N(S)\cap C\cap D\vert-\vert S\vert$, thereby contradicting the fact that $S \in \mathcal{S}$.
Thus every two vertices in $C'$ are at distance at least three from one another and so by \Cref{MostNeighbourhoodsAreCliques}, at most $(k+1)^2-1$ vertices in $C'$ do not have cliques as neighbourhoods. Denote by $C''\subset C'$ the set of vertices whose neighbourhoods are cliques. Observe that no vertex $c$ in $C''$ can be at distance two to another vertex $c'$ in $C\cap D$ as otherwise we could remove $c$ from $D$ and replace it with a common neighbour of $c$ and $c'$, yielding a minimum dominating set of $G$ containing an edge, a contradiction to \Cref{NOInstStable}. Thus, every vertex in $C\cap D$ which has a common neighbour with another vertex in $C\cap D$ must be contained in $N(S)\cap C\cap D$ or in $C'\setminus C''$, which together have cardinality at most $2|A|+(k+1)^2-3$.
\end{proof}

\begin{corollary}\label{MostVerticesAreRegular}
If $G$ is a \no-instance for \contracd then there exists a minimum dominating set $D$ of $G$ such that $\vert D\setminus\mathcal{R}\vert\leq 2(\vert A\vert +(k+1)^2)+(k+2)\vert A\vert+k-4$.
\end{corollary}
\begin{proof}
It follows from \Cref{MostVerticesInC} that there is a minimum dominating set $D$ such that $\vert D\setminus C\vert\leq (k+2)\vert A\vert +k$. Let $C_1\subset C\cap D$ be the set of the vertices in $C\cap D$ which are at distance at least three to every other vertex in $C\cap D$. Let $C_2\subseteq C_1$ be the set of the vertices in $C_1$ whose neighbourhoods are cliques. Suppose for a contradiction that there are two vertices $c,c'\in C_2$ which are at distance three. Let $b\in N(c)$ and $b'\in N(c')$ be two adjacent vertices. Then $(D\setminus\set{c,c'})\cup\set{b,b'}$ is a minimum dominating set containing an edge, a contradiction to \Cref{NOInstStable}. Thus, the vertices in $C_2$ are pairwise at distance at least four from one another. It follows that either $\vert C_2\vert\leq k$ or $C_2\subseteq\mathcal{R}$. Since by \Cref{AlmostAllAreDistanceThree} $\vert (C\cap D)\setminus C_1\vert\leq 2\vert A\vert+ (k+1)^2-3$ and by \Cref{MostNeighbourhoodsAreCliques} $\vert C_1\setminus C_2\vert\leq (k+1)^2-1$, the claim follows.
\end{proof}

We now present an algorithm which determines in polynomial time whether $G$ is a \yes-instance of \contracd or not. In the following, we let $f(k)=2(\vert A\vert +(k+1)^2)+(k+2)\vert A\vert+k-4$.

\begin{itemize}
    \item[1.] Determine $A$, $B$, $C$ and $\mathcal{R}$.
     \begin{itemize}
     \item[1.1] If $\mathcal{R}=\varnothing$, check if there exists a dominating set of size at most $f(k)$.
     \begin{itemize}
     \item[1.1.1] If the answer is no, then output \yes.
     \item[1.1.2] Else apply \Cref{boundedgamma}.
     \end{itemize}
     \item[1.2] Else go to 2.
     \end{itemize}
    \item[2.] Check whether there exist two regular vertices in $\mathcal{R}$ which are at distance at/. most three from one another. If so, output \yes.
    \item[3.] Let $V_1$ be the set of vertices at distance exactly one from $N[\mathcal{R}]$ and let $V_2 = V(G) \setminus (N[\mathcal{R}] \cup V_1)$. If $V_2 = \varnothing$, output \no.
    \item[4.] Determine $\mathcal{S}=\set{S\subseteq V_1 \cup V_2\colon \vert S\vert\leq f(k), \forall x\in V_2, N(x)\cap S\neq\varnothing}$. If $\mathcal{S}=\varnothing$, output \yes.
    \item[5.] Let $\mathcal{S}'$ be the family of all sets in $\mathcal{S}$ of minimum size. 
    \begin{itemize}
    \item[(i)] If there exists a set $S\in\mathcal{S}'$ containing an edge, output \yes. 
    \item[(ii)] If there exists a set $S\in\mathcal{S}'$ such that $S \cap V_1 \neq \varnothing$, output \yes.
    \end{itemize}
    \item[6.] Output \no.
\end{itemize}

Finally, let us show that this algorithm outputs the correct answer. In case $\mathcal{R}=\varnothing$ then by \Cref{MostVerticesAreRegular} $G$ is a \yes-instance for \contracd if there exists no dominating set of size at most $f(k)$ (see step 1.1.1). If such a set exists, then we conclude using \Cref{boundedgamma} (see step 1.1.2). If in step 2, two regular vertices at distance at most three from one another are found then by \Cref{AllRegularVerticesAtDistanceFour}, $G$ is a \yes-instance for \contracd. Otherwise, any two regular vertices are at distance at least four from one another and by \Cref{RegularNeighbourhoodShift}, there exists a minimum dominating set $D$ of $G$ such that for any regular vertex $c \in \mathcal{R}$, $D \cap N[c] = \{b_c\}$ where $b_c\in N(c) \cap B$. In the following, we denote by $D' = \bigcup_{c \in \mathcal{R}} \set{b_c}$. Note that by \Cref{VertexCompleteness}, for any $x \in N[\mathcal{R}] \cup V_1$, $N(x) \cap D' \neq \varnothing$. Now if $V_2 = \varnothing$, then we conclude by \Cref{OneVertexInDPerRegularVertex} and the fact that any two regular vertices are at distance at least four from one another, that any minimum total dominating set of $G$ is a stable set, that is, $G$ is a \no-instance for \contracd (see step 3). Otherwise $V_2 \neq \varnothing$ and if $G$ is a \no-instance for \contracd, then by \Cref{MostVerticesAreRegular} there must exist a set $S \subseteq V_1 \cup V_2$ of cardinality at most $f(k)$ such that for any $x \in V_2$, $N(x) \cap S \neq \varnothing$. Thus, if $\mathcal{S} = \varnothing$ then $G$ is a \yes-instance for \contracd (see step 4). Otherwise $\mathcal{S} \neq \varnothing$, and for any $S \in \mathcal{S}'$, $S \cup D'$ is a minimum dominating set of $G$. It then follows from \Cref{NOInstStable} that if there exists $S \in \mathcal{S}'$ such that $S$ contains an edge then $G$ is a \yes-instance for \contracd (see step 5(i)); otherwise, any $S \in \mathcal{S}'$ is a stable set and if there exists a set $S \in \mathcal{S}'$ such that $S \cap V_1 \neq \varnothing$ then $S \cup D'$ contains an edge and so, $G$ is a \yes-instance by \Cref{NOInstStable} (see step 5(ii)). Otherwise, for any $S \in \mathcal{S}'$, $S$ is a stable set and $S \cap V_1 = \varnothing$ which implies that $S \cup D'$ is a stable set and thus, $G$ is a \no-instance for \contracd. As every step can clearly be done in polynomial time, this concludes the proof of \Cref{thm:P3kP2}.

\printbibliography
\end{document}